\documentclass{article}
\usepackage[dvips]{graphicx}
\input xy
\xyoption{all}
\usepackage[dvips]{graphicx}
\usepackage{amssymb, amsmath, amsfonts, epsfig, color,amscd,amsthm}

\newtheorem{theorem}{Theorem}[section]

\newtheorem{corollary}[theorem]{Corollary}
\newtheorem{lemma}[theorem]{Lemma}
\newtheorem{proposition}[theorem]{Proposition}

\theoremstyle{definition}
\newtheorem{definition}[theorem]{Definition}
\newtheorem{example}[theorem]{Example}

\newtheorem{notation}[theorem]{Notation}
\newtheorem{remark}[theorem]{Remark}

\DeclareMathOperator{\aut}{Aut}
\DeclareMathOperator{\Ker}{Ker}
\DeclareMathOperator{\psl}{PSL}

\begin{document}

\title{Topologically Unique Maximal Elementary Abelian Group Actions on Compact Oriented Surfaces}

\author{S. Allen Broughton, Rose-Hulman Institute of Technology
\and Aaron Wootton\thanks{Partially supported by a Butine Faculty Development Grant from the University of Portland}, University of Portland }

\maketitle

\begin{center}
Keywords: Automorphism Groups of Surfaces, Moduli

Mathematical Subject Classification: 14J10, 14J50, 30F10, 30F20
\end{center}

\abstract{We determine all finite maximal elementary abelian group actions on compact oriented surfaces of genus $\sigma\geq 2$ which are unique up to topological equivalence. For certain special classes of such actions, we determine group extensions which also define unique actions. In addition, we explore in detail one of the families of such surfaces considered as compact Riemann surfaces and tackle the classical problem of constructing defining equations.}

\section{Introduction}

In \cite{BroWo}, a method was developed to calculate the number of distinct topological actions, up to topological equivalence, of a finite elementary abelian group $A$ on a compact orientable surface $S$ of genus $\sigma \geq 2$, with explicit results produced for low genus. Though in practice the method will work for any elementary abelian group $A$ and any genus, it is highly computational, often producing many different families of possible group actions. Consequently, answers to fairly straight forward follow up questions such as; ``is $A$ maximal (as a finite elementary abelian group action on $S$)?'' or, ``what larger finite groups contain $A$?'' are not really tractable. In this paper we shall investigate finite elementary abelian group actions on $S$ which are unique up to topological equivalence. By restricting to this case, we are able to determine precisely which elementary abelian groups are maximal as well as determine certain non-abelian extensions which also define unique actions. An interesting family of such surfaces are $(n-2)$-fold towers of $C_{p}$ extensions of the Fermat curves, which we call hyper-Fermat curves on which the elementary abelian group $A=C_{p}^{n}$ acts. We determine defining equations for the hyper-Fermat curves for which the action of $A$ is linear. 

Our main results are presented in three parts. Initially, we examine the problem of which elementary abelian groups that are unique up to topological equivalence are also maximal elementary abelian group actions. For this, we first prove Theorem \ref{thm-classes} which lists all groups and signatures for elementary abelian group actions that are unique up to topological equivalence and then Theorem \ref{thm-maximal} which breaks these classes into lists which define maximal actions and those which never define maximal actions. Following our analysis of the maximal actions, we concentrate on which group extensions also define unique actions (up to topological equivalence). Though generally this is a difficult problem, we are able to prove that any normal extension of a genus $0$ unique elementary abelian action also defines a unique action, see Theorem \ref{thm-pureramifiednormal}. As a consequence, it follows that any topological group acting on a surface which contains a hyperelliptic involution is unique up to topological equivalence (see for comparison \cite{Stu}). Finally, in the third part of our analysis, we discuss the families of hyper-Fermat curves in detail. As mentioned previously, the curves have a $C_{p}^{n}$-action from which the genus of a hyper-Fermat curve may be calculated to be $$\sigma =1+p^{n-1} \frac{(n-1)p+n+1}{2}.$$ By providing specific defining equations for hyper-Fermat curves, we can show that the curves depend upon $n-2$ moduli.  

Our study of finite group actions is primarily motivated by advances in the study of groups of automorphisms of compact Riemann surfaces. Due to the resolution of the Nielsen Realization Problem by Kerckhoff, see \cite{K}, every finite group of topological automorphisms of a compact surface $S$ can be realized as a finite group of conformal automorphisms of $S$ after an appropriate complex structure has been imposed on $S$. Thus the study of topological group actions can be translated into the study of conformal groups actions. In the 1960's and 1970's, a systematic study of conformal automorphism groups began using uniformizing Fuchsian groups. For instance, see the papers of Macbeath \cite{McB1}, \cite{McB2}, MacLachlan, \cite{Mac1}, \cite{Mac2}, Harvey \cite{Har1}, \cite{Har2} Gilman \cite{Gil1}, \cite{Gil2}, Gilman and Patterson \cite{GP}, and Singerman \cite{Si1}, \cite{Si2}, \cite{Si3}, \cite{Si4}.  In recent years, in part due to the advances in computer algebra systems, there has been tremendous progress in classification results of automorphism groups of compact Riemann surfaces. A review of some results as well as references are given in the paper Broughton \cite{Bro2} and the monograph Breuer \cite{Breu}. These advances coupled with the link between topological and conformal group actions has stimulated progress in previously difficult problems regarding topological group actions.

Other interesting related topics revolve around questions focused on the implications the existence of automorphism on a surface. An example of such a question would be how the existence of automorphisms affect a defining equation for a surface $S$. Such questions are of great interest, especially when the group of automorphisms is a Bely\u\i\ group, since there will always be a defining equation defined over $\bar{\mathbb Q}$. For results on such questions, see for example  \cite{Gonz}, \cite{Jon2} and \cite{Woo1}.

Further motivation for our work comes from a number of different sources. First, an understanding of the topological equivalence classes of group actions on surfaces is equivalent to an understanding of the finite subgroups of the mapping class group $\mathcal{M}_{\sigma}$ of a closed oriented surface $S$ of genus $\sigma$. This understanding is particularly important since in \cite{Mac2} it was shown that $\mathcal{M}_{\sigma}$ is generated by elements of finite order. Examining the elementary abelian actions which are unique up to topological equivalence is equivalent to examining elementary abelian subgroups of $\mathcal{M}_{\sigma}$ which are unique up to conjugacy. By restricting to maximal actions, we are imposing the further condition that it is maximal as a finite elementary abelian subgroup of $\mathcal{M}_{\sigma}$. A close analysis of these groups may provide insight into the structure of $\mathcal{M}_{\sigma}$. 

Another reason for our work is that the finite subgroups describe the singularity structure of moduli space with implications about the structure of the cohomology of the mapping class group. More specifically, $\mathcal{M}_{\sigma}$-equivariant cell complexes can be built from the singularity structure on the moduli space. Using these complexes and the moduli of these curves, one can show that if $A$ is a maximal elementary abelian subgroup of $\Gamma$, then $H^*(A;\mathbb{F}_{p})$ is a finite module over $H^*(\Gamma;\mathbb{F}_{p})$ via the restriction map  $H^*(\Gamma;\mathbb{F}_{p}) \rightarrow   H^*(A;\mathbb{F}_{p})$. Moreover, this restriction map  is ``almost injective'' in the sense that the Krull dimensions of the algebras $H^*(\Gamma;\mathbb{F}_{p} )$ and $H^*(A;\mathbb{F}_{p})$ are the same. The cases where a maximal elementary abelian subgroup is unique up to conjugacy might be interesting initial examples to study, since there would essentially be only one restriction map. See \cite{Bro3} for more background details. 

\section{\label{sec-prelim}Preliminaries}

The following is a summary of the preliminary results and notation from \cite{BroWo} which we shall adopt for our work. For a more thorough introduction, see Section 2 of \cite{BroWo}.

\begin{definition}
A finite group $G$ acts a surface $S$ if there is an embedding $\epsilon \colon G\rightarrow Homeo^{+} (S)$ where $ Homeo^{+} (S)$ denotes the group of orientation preserving homeomorphisms of $S$. We usually identify $G$ with its image.
\end{definition}

\begin{definition}
Two group actions of $G$ on $S$, defining isomorphic subgroups $G_{1}$ and $G_{2}$ of homeomorphisms of $S$, are said to be topologically equivalent if there exists $h\in Homeo^{+} (X)$ such that $G_{1} =hG_{2}h^{-1}$.
\end{definition}

\begin{definition}
We define the signature of $G$ acting on $S$ to be the tuple $(\rho;m_{1},\dots ,m_{r})$ where the orbit space $S/G$ has genus $\rho$ and the quotient map $\pi \colon S\rightarrow S/G$ is branched over $r$ points with ramification indices $m_{1},\dots, m_{r}$. We also call $\rho$ the orbit genus of the $G$-action.
\end{definition}

\begin{definition}
We say that a vector of group elements $$(a_{1},b_{1},a_{2},b_{2},\dots ,a_{\rho},b_{\rho},c_{1},\dots ,c_{r})$$ is a $(\rho;m_{1},\dots ,m_{r})$-generating vector for $G$ if the following hold:

\begin{enumerate}
\item
$G=\langle a_{1},b_{1},a_{2},b_{2},\dots ,a_{\rho},b_{\rho},c_{1},\dots ,c_{r} \rangle$.

\item
The order of $c_{i}$ is $m_{i}$ for $1\leq i\leq r$.

\item
$\prod_{i=1}^{\rho} [a_{i} ,b_{i} ] \prod_{j=1}^{r} c_{j}$=1.

\end{enumerate}
\end{definition}

\begin{remark}
Generating vectors were first introduced by Gilman in \cite{Gil1} as a tool to determine topological equivalence classes of group actions. Since then, they have been used extensively in the literature, see for example \cite{Breu}, \cite{Bro1} and \cite{BroWo}, and will likewise be used extensively in our current work.
\end{remark}

\begin{definition}
\label{definit-can}
We call a discrete subgroup $\Gamma \leq \psl{(2,{\mathbb R} )}$  a Fuchsian group with signature $(\rho;m_{1} ,\dots ,m_{r})$ if $\Gamma$  has the following presentation: $$\Gamma =\left\langle \begin{array}{c|c} A_{1},B_{1},A_{2},B_{2},\dots,A_{\rho},B_{\rho}, & C_{1}^{m_{1}},\cdots, C_{r}^{m_{r}}, \\ C_{1}%
,\dots,C_{r} & \prod_{i=1}^{\rho}[A_{i},B_{i}%
]\prod_{j=1}^{r}C_{j} \end{array} \right\rangle$$ We call any ordered set of generators $A_{1},\dots ,A_{\rho}$, $B_{1} ,\dots ,B_{\rho}$, $C_{1} ,\dots ,C_{r}$ satisfying the presentation provided by the signature a set of {\bf canonical generators} for $\Gamma$.

\end{definition}

\begin{remark}
When needed we will use $\Gamma(\rho;m_{1} ,\dots ,m_{r})$ to denote a Fuchsian group with signature $(\rho;m_{1} ,\dots ,m_{r})$. We note that the signature of a Fuchsian group determines that group up to isomorphism, though two groups with the same signature need not be conjugate in $\psl{(2,{\mathbb R} )}$. Also, as we shall see later (Lemma \ref{lemma-canonical}), any permutation of the  $m_{1},\ldots, m_{r}$ yields a valid signature of $\Gamma$.
\end{remark} 

Suppose that $v$ is a $(\rho;m_{1},\dots ,m_{r})$-generating vector for $G$ and let $\Gamma=\Gamma(\rho;m_{1} ,\dots ,m_{r})$. Then the map $\eta_{v} \colon \Gamma \rightarrow G$ defined by $\eta_{v} (A_{i})=a_{i}$, $\eta_{v} (B_{i}) =b_{i}$ and $\eta_{v} (C_{i}) =c_{i}$ is clearly an epimorphism from $\Gamma$ with signature $(\rho;m_{1},\dots ,m_{r})$ onto $G$.  We call $\eta_{v}$ an epimorphism with generating vector $v$. Alternatively, if  $\eta \colon \Gamma \rightarrow G$ is an epimorphism preserving the orders of the $C_{i}$ (we call such an epimorphism a {\bf surface kernel epimorphism}), then the vector $$v_{\eta}= (\eta (A_{1}),\eta (B_{1}) ,\eta (A_{2}), \eta (B_{2}),\dots ,\eta (A_{g}),\eta (B_{g}),\eta (C_{1}),\dots ,\eta (C_{r}))$$ is a $(\rho;m_{1},\dots ,m_{r})$-generating vector for $G$.  It follows that there is a natural action of the group $\aut{(G)}\times \aut{(\Gamma )}$ on the set of $(\rho;m_{1},\dots ,m_{r})$-generating vectors of $G$. Specifically, if $v$ is a generating vector and $\eta_{v}$ is the epimorphism with generating vector $v$, and $(\alpha ,\gamma )\in \aut{(G)}\times \aut{(\Gamma )}$, then we define $$(\alpha ,\gamma ) \cdot v_{\eta}=v_{\alpha \circ \eta \circ \gamma^{-1}}.$$   In diagram form,  the  $\aut{(G)}\times \aut{(\Gamma )}$ action is depicted in Figure \ref{figure-action}.

\begin{figure}[h]
\label{figure-groups}

\begin{center}
$
\xymatrix @R=.4in @C=.8in {
\Gamma \ar[r]^{\eta_{v}} \ar[d]_{\gamma}  &  G \ar[d]^{\alpha}  \\
\Gamma \ar[r]^{\alpha \circ \eta_{v} \circ \gamma^{-1}} &  G   \\
}
$
\end{center}

\caption{\label{figure-action}$\aut{(G)}\times \aut{(\Gamma )}$-action}

\end{figure}

The following gives us a way to enumerate topological equivalence classes of group actions using epimorphisms and generating vectors, see \cite{Bro1} for details.

\begin{theorem}
There is a one-one correspondence between $\aut{(G)} \times \aut{(\Gamma )}$ classes of $(\rho;m_{1},\dots ,m_{r})$-generating vectors of a finite group $G$ and the topological equivalence classes of $(\rho;m_{1},\dots ,m_{r})$-actions of $G$  on a surface $S$ with genus $$\sigma =1+|G|(\rho-1) +\frac{|G|}{2} \sum_{i=1}^{r} \bigg( 1-\frac{1}{m_{i}} \bigg).$$

\end{theorem}

Starting with a conformal  $(\rho;m_{1},\dots ,m_{r})$-action of  $G$ on $S$, the group $\Gamma$ may be constructed from the pair $(T,\mathcal{B})$, where $T=S/G$ and $\mathcal{B}$
is the set of branch points with branch order taken into account. Now suppose  $\Gamma$ is chosen and an epimorphism $\eta : \Gamma \rightarrow G$ is given and let  $\Pi=\Ker{(\eta)}$. Then the natural action of  $G\simeq\Gamma/\Pi$ on $S=\mathbb{H}/\Pi$ is a representative of the corresponding topological class of actions. Given a fixed $\Gamma$ (or alternatively a conformal equivalence class of quotient pairs $(T,\mathcal{B})$); there are only a finite number of surfaces $S_{1},\ldots,S_{e}$ that may be so constructed, since there are only finitely many epimorphisms $\eta : \Gamma \rightarrow G$. We describe this situation by saying  the surfaces $S_{1},\ldots,S_{e}$ lie above $(T,\mathcal{B})$. The induced actions of $G$ on two surfaces lying above $(T,\mathcal{B})$ are topologically equivalent if the corresponding kernels are equivalent by the  automorphism $\gamma$ of $\Gamma$ in the left vertical map in Figure \ref{figure-action}.  The two surfaces will be conformally equivalent if the automorphism $\gamma$ is induced by an automorphism of  $\mathbb{H}$ normalizing $\Gamma$. Finally, there will be a unique surface $S$ lying above  $(T,\mathcal{B})$ if all generating vectors with the given signature are equivalent under $\aut{(G)}$.  In this case, every conformal automorphism of $(T,\mathcal{B})$ lifts to a conformal automorphism of  $S$ normalizing the $G$-action on $S$. We will use this condition in Sections \ref{sec-normal} and \ref{sec-hyperfermat}, and for future reference we state it as a proposition.

\begin{proposition}
\label{prop-uniquevector}
Suppose that $G$ acts on two surfaces  $S_{1}$ and $S_{2}$ with signature  $(\rho;m_{1},\dots ,m_{r})$ and that all $(\rho;m_{1},\dots ,m_{r})$-generating vectors of $G$ are $\aut{(G)}$-equivalent. Then $S_{1}$ and $S_{1}$ are conformally equivalent with conformally equivalent $G$-actions if and only if $S_{1}/G$ and $S_{2}/G$ are conformally equivalent, respecting branch sets and branching orders.

\end{proposition}

\begin{notation}
If the number of branch points  $r=0$ then the action of $G$ is fixed point free and so we call the action {\em unramified} and in this case, the signature is denoted by $(\rho;-)$. If the orbit genus $\rho=0$ then the $G$-action is generated by elements with fixed points and so we call the action {\em purely ramified}.
\end{notation}

\section{Determination of Unique Classes}

For a fixed genus $\sigma \geq 2$, our first task is to determine each elementary abelian group $A$ of $p$-rank $n$ (or equivalently of order $p^n$) which can act on a surface with fixed signature $(\rho;p^r)$ for which there is a unique action up to topological equivalence. First note that for such a group to act on a surface of genus $\sigma \geq 2$, the integers $\sigma$, $n$, $\rho$ and $p$ must satisfy the Riemann-Hurwitz formula: $$\sigma=1+p^n(\rho-1)+\frac{p^{n-1}r(p-1)}{2}.$$ If this equation is satisfied, then we can calculate the total number of equivalence classes of $A$-actions with signature $(\rho;p^r)$ using the results developed in \cite{BroWo}. Specifically, \cite[Corollary 2.9]{BroWo} states that we proceed as follows:

\begin{enumerate}

\item
If $n>2\rho+r-1$ or $r=1$, there are no such actions. Else we proceed as follows.

\item
(Purely ramified actions) For each $p$-rank $1 \leq k\leq r-1$ determine number of classes of $(0;p^r)$-generating vectors in an elementary abelian group of $p$-rank $k$.  Denote this number by $e_{k}$. We also define $e_{0}= 1$ if $r=0$, $e_{0}=0$ otherwise and $e_{k}=0$ for any $k<0$.

\item
(Unramified actions) For each $p$-rank $0\leq k\leq n$ determine number of $(\rho;-)$-generating vectors in an elementary abelian group of $p$-rank $k$. Denote this number by $h_{k}$.

\item
The total number of topological equivalence classes of group actions of an elementary abelian group $A$ of $p$-rank $n$ with signature $(\rho;p^r)$ on a surface of genus  $\sigma$ is given by the sum $$\sum_{k=0}^{n} h_{k} e_{r-(k+1)}.$$

\end{enumerate}

\noindent
Following the method outlined above, to determine which groups and signatures give unique classes, we simply determine which  purely ramified and which unramified actions give unique classes.  The different classes for the unramified case were first determined in \cite{Bro2}. In particular, we have the following.

\begin{proposition}
\label{prop-number}
Suppose an elementary abelian group $A$ of $p$-rank $n$ acts on a surface $S$ of genus $\sigma \geq 2$ with signature $(\rho;-)$. Then this action is unique up to topological equivalence if and only if either $n=0$, $1$, $\rho-1$ or $\rho$.

\end{proposition}

Next we consider the purely ramified case which requires a little more work. We do not show that the stated cases produce unique classes since this can be determined using the results of \cite{BroWo}. Instead we show explicitly that for the remaining groups and signatures, there always exists at least two different classes.

\begin{proposition}
\label{prop-unique}
Suppose an elementary abelian group $A$ of $p$-rank $n$ acts on a surface $S$ of genus $\sigma \geq 2$ with signature $(0;p^r)$. Then this action is unique up to topological equivalence if and only if $p$, $r$ and $n$ satisfy one of the cases in Table \ref{Tab-PureRam}.

\begin{table}[h]
\begin{center}
$
\begin{array}{||c|c|c|c||c|c|c|c||}

\hline \hline

\text{Case} & r & n & p & \text{Case} & r & n & p  \\

\hline \hline

1 & r \text{ even} & 1 & 2 & 5 & 3 & 1 & 5 \\

\hline

2 & 5 & 3 & 2 & 6 & 2  &  1 & p  \\

\hline

3 & 4,5 & 2 & 2 & 7 & n+1 & n & p \\

\hline

4 & 3,4,5,7 & 1 & 3 &  &  &  &  \\

\hline \hline

\end{array}
$
\end{center}
\caption{\label{Tab-PureRam}Topologically Unique Purely Ramified Group Actions}

\end{table}

\end{proposition}

\begin{proof}
For convenience of notation, we shall consider $\aut{(A)} \times \aut{(\Gamma )}$-classes of surface kernel epimorphisms from a Fuchsian group $\Gamma$ with signature $(0;p^r)$ onto $A$ with $p$-rank $n$ and generators $X_{1},\dots ,X_{n}$. Our method of proof is to explicitly construct two inequivalent epimorphisms for all the cases not listed in the statement of the theorem. In order to distinguish equivalence classes  we introduce an  $\aut{(A)} \times \aut{(\Gamma )}$-invariant, $ \chi(\eta)$, of an epimorphism $\eta$, called the multi-set character (of the image). To this end, observe that the action of $\aut{(\Gamma)}$ on an epimorphism $\eta$ is a permutation of the images of the generators of $\Gamma$ under $\eta$, see \cite[Proposition 2.6]{BroWo} . In particular, the action of $\aut{(G)} \times \aut{(\Gamma )}$ will not change the number of distinct images of generators nor the number of repeated images of generators under $\eta$ (though it may change which images are repeated). More precisely, given $\eta$, we define integers  $s$ and $e_{1}\leq \cdots \leq e_{s}$ so that the multi-set with repetition, $\{\eta(C_{1}),\ldots,\eta(C_{r})\}$, consists of the $s$ distinct images  $\{g_{1},\ldots,g_{s}\}$  with $g_{i}$ repeated $e_{i}$ times, $1\leq i \leq s$.   We define $\chi(\eta)=(e_{1}, \ldots ,e_{s})$. By the discussion above two epimorphisms with different multi-set character cannot be equivalent.

We consider five main cases $n \geq 3$; $n=2$, $p\neq 2$;  $n=2$, $p= 2$;  $n=1$, $p>3$; and $n=1$, $p=3$.

First suppose that $n\geq 3$, we may suppose that $r>n+1$, since $r\geq n+1$ and $r=n+1$ yields a unique class. We may then define two epimorphisms $\eta_{1}$ and $\eta_{2}$ as follows: $$\eta_{1} (C_{i} ) = \begin{cases} X_{i} & 1\leq i\leq n \\ X_{1} & n+1\leq i\leq r-1 \\ (X_{1}^{r-n} X_{2} \dots X_{n})^{-1} & i=r \end{cases} $$  $$\eta_{2} (C_{i} ) = \begin{cases} X_{i} & 1\leq i\leq n \\ X_{1}X_{2} & n+1\leq i\leq r-1 \\ (X_{1}^{r-n} X_{2}^{r-n} \dots X_{n})^{-1} & i=r \end{cases} $$
We consider two subcases. \newline
\textit{Case $p\neq 2,$ or $p=2$, $n>3$}. Since $ (X_{1}^{r-n} X_{2} \dots X_{n})^{-1}$ is distinct from $X_{1},X_{2}, \dots ,X_{n}$, the multi-set character $\chi(\eta_{1})=(1^{n},r-n)$, namely $n$ distinct singleton values  $\{X_{2},\ldots,X_{n}, (X_{1}^{r-n} X_{2} \dots X_{n})^{-1}\}$  and $X_{1}$ repeated $r-n$ times, a total $n+1$ distinct images. For $\eta_{2}$,  since $ (X_{1}^{r-n} X_{2}^{r-n} \cdots X_{n})^{-1}$ is distinct from $X_{1}$, $X_{2}, \dots ,X_{n}$ and $X_{1}X_{2}$, then $\chi(\eta_{2})=(1^{n+1},r-n-1)$ if $r\geq n+2$ and $\chi(\eta_{2})=(1^{n+2})$ if $r= n+2$. In either case there are $n+2$ distinct images and $\chi(\eta_{2})\ne \chi(\eta_{1})$.  \newline
\textit{Case $p=2$, $n=3$, $r>5$.} We still have $\chi(\eta_{1})=(1^{n},r-n)$, but $\chi(\eta_{2})$ depends on the the parity of  $r$. If $r$ is odd   $X_{1}^{r-n}X_{2}^{r-n} X_{3} = X_{3}$ and $\chi(\eta_{2})=(1^4,r-4)\ne \chi(\eta_{1})$.  If $r$ is even  $X_{1}^{r-n}X_{2}^{r-n} X_{3} \neq X_{3}$ and $\chi(\eta_{2})=(1^2,2, r-4)\ne \chi(\eta_{1})$.

Now suppose that $n= 2$ and $p\neq 2$. Similar to the previous case, we define maps $\eta_{1}$ and $\eta_{2}$ as:
$$\eta_{1} (C_{i} ) = \begin{cases} X_{i} & i=1,2 \\ X_{1} & 3\leq i\leq r-1 \\ (X_{1}^{r-2} X_{2})^{-1} & i=r \end{cases} $$
$$\eta_{2} (C_{i} ) = \begin{cases} X_{i} & i=1,2 \\ X_{1}X_{2} & 3\leq i\leq r-1 \\ (X_{1}^{r-2} X_{2}^{r-2})^{-1} & i=r \end{cases} $$
In the three cases below, since $p\neq 2$, $(X_{1}^{r-2} X_{2})^{-1}$ is distinct from $X_{1}$ and $X_{2}$, so that  $\chi(\eta_{1})= (1^2,r-2)$ \newline
\textit{Case $p\neq 2$, $n=2$, $p\nmid (r-2)$ and $p\nmid (r-1)$}.
Since $X_{1}$, $X_{2}$, $X_{1}X_{2}$, and $(X_{1}^{r-2} X_{2}^{r-2} )^{-1}$ are all distinct, then $\chi(\eta_{2})=(1^3,r-3)\ne \chi(\eta_{1})$. \newline
\textit{Case $p\neq 2$, $n=2$, $p|(r-2)$}. The map $\eta_{2}$ is not a surface kernel epimorphism since the image of the last generator will be trivial. In this case, we redefine $$\eta_{2} (C_{i} ) = \begin{cases} X_{i} & i=1,2 \\ X_{1}^2X_{2}^2 & 3\leq i\leq r-1 \\ X_{1}X_{2} & i=r \end{cases}. $$ Since $X_{1}$, $X_{2}$, $X_{1}^2X_{2}^2$, and $X_{1}X_{2}$ are all distinct, we still have $\chi(\eta_{2})=(1^3,r-3)\ne \chi(\eta_{1})$.\newline
\textit{Case $p\neq 2$, $n=2$, $p|(r-1)$}. In this case $\eta_{1}$ and $\eta_{2}$ are in fact equivalent, so we redefine  $\eta_{2}$ as
$$\eta_{2} (C_{i} ) = \begin{cases} X_{i} & i=1,2 \\ X_{1}X_{2} & 3\leq i\leq r-2 \\ X_{1} & i=r-1 \\ X_{1}X_{2}^{2} & i=r \end{cases} $$
and we have $\chi(\eta_{2})=(1^3,r-3)\neq \chi(\eta_{1})$ provided $r>4$.  If $r=4$ then we must have $p=3$, and the epimorphism $\eta_{2}$ defined by $$\eta_{2} (C_{i} ) = \begin{cases} X_{i} & i=1,2 \\ X_{i-2}^{-1} & i=3,4 \end{cases} $$ satisfies $\chi(\eta_{2})=(1^4)\neq \chi(\eta_{1})$ (see also \cite[Example 38]{BroWo}).

For $n=2$, the last case we need to consider is $p=2$. We know that $r>5$ since $r=4$ and $5$ give unique classes of epimorphisms. Therefore, for $r$ even the epimorphisms $\eta_{1}$ and $\eta_{2}$ defined by
$$\eta_{1} (C_{i} ) = \begin{cases} X_{1} & i=1,2 \\ X_{2} & i\geq 2 \end{cases} $$
$$\eta_{2} (C_{i} ) = \begin{cases} X_{1} & i=1,2 \\ X_{2} & i=3,4 \\ X_{1} X_{2} & i>4 \end{cases} $$ define inequivalent epimorphisms. Likewise,  if $r$ is odd, the epimorphisms $\eta_{1}$ and $\eta_{2}$ defined by
$$\eta_{1} (C_{i} ) = \begin{cases} X_{1} & i=1,2,3 \\ X_{2} & i=4 \\ X_{1} X_{2} & i\geq 5\end{cases} $$
$$\eta_{2} (C_{i} ) = \begin{cases} X_{1} & i=1 \\ X_{2} & i=2 \\ X_{1} X_{2} & i>2 \end{cases} $$ define inequivalent epimorphisms since $r>5$.

Next, we need to consider the case when $n=1$ and $p>3$. \newline
\textit{Case $p> 3$, $n=1$, $r\geq 3$, $p\nmid(r-1)$, $p\nmid(r)$, $p\nmid(r+1)$}.
We define inequivalent epimorphisms $\eta_{1}$ and $\eta_{2}$ as
$$\eta_{1} (C_{i} ) = \begin{cases} X_{1} & i=1 ,\dots r-1 \\ (X_{1}^{r-1})^{-1} & i=r \end{cases} $$
$$\eta_{2} (C_{i} ) = \begin{cases} X_{1} & 1,\dots r-2 \\ X_{1}^{2} & i=r-1 \\ (X_{1}^{r})^{-1} & i=r\end{cases} $$
Since  $p\nmid(r-1)$, $p\nmid(r)$, both epimorphisms are defined. We always have $\chi(\eta_{1})=(1,r-1)$.  For $\eta_{2}$ we have  $\chi(\eta_{2})=(1^2,r-2)$ or $(1^3)$ if $p\nmid(r+2)$ and $\chi(\eta_{2})=(2,r-2)$ if $p|(r+2)$ except in the cases $p=5,r=3$ and $p=2,3, r=4$. These excluded cases are listed in Table \ref{Tab-PureRam}. \newline
\textit{Case $p> 3$, $n=1$, $r\geq 3$, $p|(r-1)$.}
If $p|(r-1)$, then the map $\eta_{1}$ is not a surface kernel epimorphism, so we define an epimorphism $\eta_{1}$ which is not equivalent to $\eta_{2}$ by $$\eta_{1} (C_{i} ) = \begin{cases} X_{1} & 1,\dots r -3 \\ X_{1}^{2} & i=r-1, r-2 \\ X_{1}^{-2} & i=r \end{cases} $$ Note that provided $r\geq 4$, this will define an epimorphism which is inequivalent to $\eta_{2}$ and in the case $r=3$, we must have $p=2$.\newline
\textit{Case $p> 3$, $n=1$, $r\geq 3$, $p|r$.}
If $p|r$, then $\eta_{2}$ is not a surface kernel epimorphism, so instead we define an epimorphism $\eta_{2}$ which is not equivalent to $\eta_{1}$ by  $$\eta_{2} (C_{i} ) = \begin{cases} X_{1} & 1,\dots r -2 \\ X_{1}^{4} & i=r-1 \\ (X_{1}^{2})^{-1} & i=r \end{cases} $$ \newline
\textit{Case $p> 3$, $n=1$, $p|(r+1)$}.
Since $p>3$ and $p|(r+1)$, $p$ cannot divide $r-1$, $r$, $r+1$, or $r+2$. We redefine  $\eta_{2}$ by   $$\eta_{2} (C_{i} ) = \begin{cases} X_{1} & i=1,\dots r-2 \\ X_{1}^{4} & i=r-1 \\ (X_{1}^{r+2})^{-1} & i=r\end{cases} $$ Now $\chi(\eta_{2})=(1^2,r-2)\neq \chi(\eta_{1}).$  \newline

The last case we need to consider is when $p=3$ and $n=1$
We set up epimorphisms $\eta_{1}$ and $\eta_{2}$ depending upon $r\mod{(3)}$.
If $r\equiv 0\mod{(3)}$, then we define  $\eta_{1}$ and $\eta_{2}$ as
$$\eta_{1} (C_{i} ) = \begin{cases} X_{1} & i=1,\dots, r \end{cases} $$
$$\eta_{2} (C_{i} ) = \begin{cases} X_{1}^2 & i=1,\dots,r-3 \\ X_{1} & i=r-2,\dots ,r \end{cases} $$
which are inequivalent provided $r>3$.
If $r\equiv 1\mod{(3)}$, then we define  $\eta_{1}$ and $\eta_{2}$ as
$$\eta_{1} (C_{i} ) = \begin{cases} X_{1} & 1,\dots, r  -2 \\ X_{1}^{2} & i=r-1,r \end{cases} $$
$$\eta_{2} (C_{i} ) = \begin{cases} X_{1} & 1,\dots, r -5 \\ X_{1}^{2} & i=r-4,\dots ,r \end{cases} $$
which are inequivalent provided $r>7$. For $r\equiv 2\mod{(3)}$, we define  $\eta_{1}$ and $\eta_{2}$ as
$$\eta_{1} (C_{i} ) = \begin{cases} X_{1} & 1,\dots, r  -1 \\ X_{1}^{2} & i=r \end{cases} $$
$$\eta_{2} (C_{i} ) = \begin{cases} X_{1} & 1,\dots, r -4 \\ X_{1}^{2} & i=r-3,\dots ,r \end{cases} $$ which are inequivalent provided $r>5$.
\end{proof}

We can now use these results to determine the groups for unique classes.

\begin{theorem}
\label{thm-classes}
Suppose an elementary abelian group $A$ of $p$-rank $n$ acts on a surface $S$ of genus $\sigma \geq 2$ with signature $(\rho;p^r)$ where $$\sigma =1+p^n (\rho -1) +\frac{rp^{n-1}(p-1)}{2}.$$ Then this action is unique up to topological equivalence if and only if $\rho$, $p$, $r$ and $n$ satisfy one of the cases in Table \ref{Tab-Unique}.

\begin{table}[h]

\begin{center}

\[
\begin{array}{||c|c|c||c|c|c||}

\hline \hline
\text{Case} & \text{Signature} & \text{Conditions} & \text{Case} & \text{Signature} & \text{Conditions} \\

\hline \hline

1 & (0;p^r) & n=r-1 & 8 & (\rho;3^3) & n=1 \\
\hline
2 & (\rho;-) & n=1 & 9 & (\rho;3^4) & n=1 \\
\hline
3 & (\rho;-) & n=2\rho & 10 & (\rho;3^5) & n=1  \\
\hline
4 & (\rho;p^2) & n=1, \rho\geq 1 & 11 & (\rho;3^7) & n=1  \\
\hline
5 & (\rho;p^r) & n=r+2\rho-1 & 12 & (\rho;-) & n=2\rho-1 \\
\hline
6 & (\rho;5^3) & n=1 & 13 & (0;2^5) & n=3 \\
\hline
7 & (\rho;2^r) & n=1, r \text{ even} & 14 & (\rho;2^5) & n=2 \\
\hline \hline

\end{array}
\]

\end{center}

\caption{\label{Tab-Unique}Unique Group Actions}

\end{table}

\end{theorem}

\begin{proof}
This is a consequence of \cite[Corollary 2.9]{BroWo} and Propositions \ref{prop-number} and \ref{prop-unique}. Specifically, we only need to restrict ourselves to signatures for purely ramified and unramified for which there is a unique class of epimorphism and then build the possible combinations.
\end{proof}

\section{Maximal Actions}

Suppose that $A$ of $p$-rank $n$ and signature $(\rho;p^r)$ appears in Table \ref{Tab-Unique}. If $A$ is not maximal (as an elementary abelian group), then there exists a group $N$ of $p$-rank $n+1$ and signature $(\tau;p^s)$ such that $A\leq N$. In order to determine whether $A$ is maximal, we shall determine whether or not such a group $N$ can exist. First, we need the following result which allows us to determine the signature of any subgroup $A$ of an elementary abelian group $N$ acting on $S$ with signature $(\tau;p^s)$.

\begin{proposition}
\label{prop-RH}
Suppose that the elementary abelian group $N$ acts on a surface $S$ of genus $\sigma \geq 2$ with signature $(\tau;p^s)$ and generating vector $(a_{1},b_{1}\dots ,a_{\tau},b_{\tau}, c_{1},\dots ,c_{s})$. If $A$ is a subgroup of $N$ and $\chi \colon N\rightarrow N/A$ is the quotient map, then the signature of $A$ acting on $S$  is $(\rho;p^{|N|m/|A|} )$ where $l=s-m$ is the number of $c_{i}$ which have non-trivial image under $\chi$ and $$\rho-1=\frac{|N|}{|A|} (\tau-1) +\frac{|N|}{2|A|} \sum_{i=1}^{l} \bigg( 1-\frac{1}{p} \bigg).$$

\end{proposition}

\begin{proof}
This is just a special case of the more general result for normal subgroups of an arbitrary group acting on $S$, see for example Lemma 3.6 of \cite{Breu}.
\end{proof}

In the special case where $A$ has $p$-rank $n$ and  $N$ has $p$-rank $n+1$, we have the following useful consequences.

\begin{corollary}
\label{cor-RH}
There exists integers $l$ and $m$ such that $s=l+m$, $r=pm$ and \begin{equation} \label{eqn-RH} 2\rho-2=2p(\tau-1)+l(p-1). \end{equation}
\end{corollary}

\begin{corollary}
\label{cor-simfac}
The following must be true:

\begin{enumerate}
\item
$\tau\leq \rho$ with equality only if $\rho=0$ or $\rho=1$.

\item
If $\rho=0$, then $r>s$.

\item
$p|r$.

\item
If $\rho=1$, then $l=0$.

\end{enumerate}
\end{corollary}

\begin{theorem}
\label{thm-maximal}
The choices of signature $(\rho ;p^r)$ and positive integer $n$ for which there is a unique elementary abelian action of $p$-rank $n$ with signature $(\rho;p^r)$ on a surface of genus $$\sigma =1+p^n (\rho -1) +\frac{rp^{n-1}(p-1)}{2} $$ which is always maximal are given Table 3.

\begin{table}[h]

\begin{center}

$
\begin{array}{||c|c|c||c|c|c||}

\hline \hline
\text{Case} & \text{Signature} & \text{Conditions} & \text{Case} & \text{Signature} & \text{Conditions} \\

\hline \hline

1 & (0;p^r) & n=r-1 & 7 & (\rho;3^4) & n=1 \\

\hline

2 & (\rho;-) & n=2\rho, p\neq 2 & 8 & (\rho;3^5) & n=1 \\

\hline

3 & (\rho;p^2) & n=1, p\neq 2 & 9 & (\rho;3^7) & n=1  \\

\hline

4 & (\rho;p^r) & n=r+2\rho-1,  pr\neq 4 & 10 & (0;2^5) & n=3  \\

\hline

5 & (\rho;5^3) & n=1 & 11 & (\rho;2^5) & n=2 \\

\hline

6 & (\rho;3^4) & n=1 & & &  \\

\hline \hline

\end{array}
$

\end{center}

\caption{\label{Tab-Max}Maximal Unique Group Actions}

\end{table}

\noindent The choices of signature for which $A$ is never maximal are given Table 4.

\begin{table}[h]

\begin{center}

$
\begin{array}{||c|c|c||c|c|c||}

\hline \hline
\text{Case} & \text{Signature} & \text{Conditions} & \text{Case} & \text{Signature} & \text{Conditions} \\

\hline \hline

1 & (\rho;-) & n=2\rho, p=2 & 4 & (\rho;2^r) & n=1, r \text{ even} \\

\hline

2 & (\rho;2^2) & n=1 & 5 & (\rho;3^3) & n=1 \\

\hline

3 & (\rho;2^2) & n=2\rho-1 & 6 & (\rho;-) & n=2\rho-1, p=2  \\

\hline \hline

\end{array}
$

\end{center}

\caption{\label{Tab-NotMax}Non-Maximal Unique Group Actions}

\end{table}

\noindent Finally, for $(\rho;-)$ with $n=1$ and $\rho \geq 2$, if $p=2$, then $A$ is never maximal. If $p>2$, then $A$ is not maximal if and only if $\rho$ satisfies $\rho=ap+b(p-1)/2 +1$,  for integers $a$ and $b$ with $a\geq -1$, $b\geq 0$. In particular, for a given $p$, there are only finitely many values of $\rho$ for which this group is maximal (by the Fr\"obenius problem with $p$ and $(p-1)/2$).

\end{theorem}

\begin{proof}
We refer to the cases in Theorem \ref{thm-classes}, listed in Table 2. First, by $(3)$ of Corollary \ref{cor-simfac}, Cases $6$, $9$, $10$, $11$ $13$ and $14$ must define maximal actions, and for the same reason, provided $p\neq 2$, Case $4$ also defines a maximal action. If $p=2$, then the action is never maximal. Specifically, if $C_{2} \times C_{2} =\langle x,y\rangle$, then if $\rho$ is odd, there is a $C_{2}\times C_{2}$-action with signature $(\frac{\rho-1}{2};2,2,2,2,2)$ and generating vector $(e,\dots ,e,x,x,x,xy,y)$ extending the action of $C_{2} =\langle y \rangle$ with signature $(\rho;2,2)$, and if $\rho$ is even, a $C_{2}\times C_{2}$-action with signature $(\frac{\rho}{2};2,2,2)$ and corresponding generating vector $(e,\dots ,e,x,xy,y)$ extending the action of $C_{2} =\langle y \rangle$  with signature $(\rho;2,2)$ ($e$ denotes the identity of $A$).

For Case $1$, the maximality is a direct consequence of $(2)$ of Corollary \ref{cor-simfac}. Specifically, if $A$ has signature $p^r$ and $p$-rank $n=r-1$ and $N$ is an extension by $C_{p}$, then it will have $p$-rank $n+1$ and signature $(0;p^k)$ where $k<r$. However, the minimal number of elements required to generate a $p$-rank $n+1$ group is $n+2>k$.

Both Cases $7$ and $8$ define signatures which are never maximal. Specifically, for Case $7$, if $C_{2} \times C_{2} =\langle x,y\rangle$ then we have a $C_{2} \times C_{2}$-action with signature $(0;2^r)$ where $r=k+2(\rho+1)$ with corresponding generating vector $(\underbrace{y\dots ,y}_{k }, x, xy, \underbrace{x\dots ,x}_{2\rho })$ if $k$ is odd and $(\underbrace{y\dots ,y}_{k }, \underbrace{x,\dots ,x}_{2(\rho+1) })$ if $k$ is even extending the action of $A=\langle y\rangle$ with signature $(\rho;2^{2k})$. For Case $8$,  if $C_{3} \times C_{3} =\langle x,y\rangle$ then we have a $C_{3} \times C_{3}$-action with signature $(0;3^r)$ where $r=\rho+3$ with corresponding generating vector $(y,xy^{-1},x^{-1},x,\dots ,x)$ if $\rho\equiv 0\mod{(3)}$, $(y,xy,xy,x,\dots ,x)$ if $\rho\equiv 1\mod{(3)}$, and $(y,x^{-1}y,x^{-1}y,x,\dots ,x)$ if $\rho \equiv 2\mod{(3)}$, extending the action of $A=\langle y\rangle$ with signature $(\rho;3^{3})$.

The arguments for Cases $3$ and $12$ are similar, so we only provide details for case $12$, the more technical of the two. If $N$ is a $C_{p}$ extension with signature $(\tau;p^k)$, then using Corollary \ref{cor-RH}, we have $$2\rho-2=2p(\tau-1)+k(p-1)$$ (since the kernel is torsion free, all the elliptic generators must have non-trivial image under $\chi \circ \eta$). By assumption, $N$ has $p$-rank $2\rho$ (since $A$ has rank $2\rho-1$). However, through observation of its signature, the largest rank $N$ could have is $2\tau+k-1$. Thus we must have $$2\rho=2p(\tau-1)+k(p-1)+2\leq 2\tau+k-1.$$ Simplifying, we get $$2\tau(p-1)+k(p-2)+3\leq 2p.$$ If $\tau\geq 1$, since $p\geq 2$ we get $$2p+k(p-2)+1 \leq 2\tau(p-1)+k(p-2)+3\leq 2p$$ or $k(p-2) \leq -1$ which is absurd, so we must have $\tau=0$. When $\tau=0$, we must have $k\geq 3$ and thus we get  $3p-3 \leq k(p-2)+3\leq 2p$ or $p\leq 3$. If $p=3$, there is no choice of $k$ so that $(0;3^k)$ is a $C_{3}$-extension of $A$ with $3$-rank $2\rho-1$ and thus $A$ is maximal. If $p=2$, and $N=\langle x_{1},\dots ,x_{2\rho}\rangle$ then $N$ with signature $(0;2^{2\rho+2})$ and corresponding generating vector $(x_{1},x_{2} ,\dots ,x_{2\rho}, x_{1},x_{2}x_{3}x_{4}x_{5} \dots x_{2\rho})$ defines a $C_{2}$ extension of $A=\langle x_{1}x_{2},x_{1}x_{3},\dots ,x_{1}x_{2\rho-1} \rangle$ with signature $(\rho;-)$. Thus $A$ with $2$-rank $2\rho-1$ and signature $(\rho;0)$ is never maximal. The same results holds for Case $3$.

For Case $5$, the argument is similar to the previous two cases. If $N$ is a $C_{p}$ extension with signature $(\tau;p^k)$, then using Corollary \ref{cor-RH}, we must have $$2\rho-2=2p(\tau-1)+l(p-1).$$  By assumption, $N$ has $p$-rank $2\rho+r$. However, through observation of its signature, the largest rank $N$ could have is $2\tau+k-1$. Thus we must have $$2\rho+r=2p(\tau-1)+l(p-1)+r+2\leq 2\tau+k-1.$$ Now observe that $k=l+m$ and $r=pm$, so we have $$2\rho+r=2p(\tau-1)+l(p-1)+pm+2\leq 2\tau+k-1=2\tau+l+m-1$$ which after simplification becomes $$2\tau(p-1)+l(p-2)+m(p-1)+3 \leq 2p.$$ Imitating our proof above, we must have $\tau=0$, in which case we get $$l(p-2)+m(p-1)+3 \leq 2p.$$ When $\tau=0$, we have  $$l(p-2)+m(p-1)+3 \leq 2p$$ giving $$p\leq 1+\frac{l-1}{l+m-2}.$$ Since $m\geq 1$ (else this reduces to Case $3$), it follows that $p\leq 2$, so the only remaining case to examine is when $\tau=0$ and $p=2$. Observe though that $p=2$ only when $m=1$ and $A$ has signature $(\rho;2,2)$. In this case however $N=\langle x_{1},\dots ,x_{2\rho+2}\rangle$ with signature $(0;2^{2\rho+3})$ and corresponding generating vector $(x_{1},x_{2},x_{3} ,\dots ,x_{2\rho-1}, x_{1}x_{2}x_{2\rho+2}, x_{3}x_{4}x_{5} \dots x_{2\rho+2)}$ defines a $C_{2}$ extension of $A=\langle x_{1}x_{2},x_{1}x_{3},\dots ,x_{1}x_{2\rho+2} \rangle$ with signature $(\rho;2,2)$. Thus $A$ with $2$-rank $2\rho+1$ and signature $(\rho;2,2)$ is never maximal.

Finally, we examine Case $2$. First note that, if $p=2$, the group $C_{2}\times C_{2}=\langle x,y\rangle$ with signature $(1;2^{k})$ and corresponding generating vector $(y,y,y,y,x,x,\dots ,x)$ where $k=2(\rho-1)$ is always a $C_{2}$ extension of $A=\langle y\rangle$ with signature $(\rho;0)$. Now suppose that $p\neq 2$. If $A$ is not maximal, then there exists $N$ of $p$-rank $2$ with signature $(\tau;p^k)$ which extends $A$ with signature $(\rho;-)$. First note, that if a $C_{p}$-normal extension of $C_{p}$ with signature $(\rho;-)$ by $N$ with signature $(\tau;p^k)$ exists, then $\rho\geq 2$ satisfies the Riemann-Hurwitz formula, $\rho=ap+b(p-1)/2 +1$, for integers $a$ and $b$ with $a\geq -1$, $b\geq 0$. We shall show that this condition is in fact sufficient. Suppose that $(\tau;p^k)$ satisfies the Riemann-Hurwitz formula and let $N=C_{p} \times C_{p} =\langle x,y\rangle$. Then $N$ with generating vector $(\underbrace{e\dots ,e}_{2\tau \text{ times }},\underbrace{x\dots ,x}_{k-2 \text{ times }},x^{-1}y,x^{2}y^{-1})$ if $p\equiv 1\mod{(p)}$ and generating vector  $(\underbrace{e\dots ,e}_{2\tau \text{ times }},\underbrace{x\dots ,x}_{k-2 \text{ times }},xy,(x^{k-1}y)^{-1})$ else both define extensions of $A=\langle y\rangle$ with signature $(\rho;0)$ provided $k>2$. If $k=2$, then $N$ with generating vector $(\underbrace{xy\dots ,xy}_{2\tau \text{ times }},x,x^{-1})$ defines a $C_{p}$ extension of $A=\langle y\rangle$ with signature $(\rho;-)$.
\end{proof}

\section{\label{sec-normal}Normal Group Extensions of Genus $0$ Groups}

Theorem \ref{thm-classes} provides all the possible signatures for which there exists a unique topological equivalence class of elementary abelian groups of homeomorphisms of a surface of genus $\sigma$ and Theorem \ref{thm-maximal} provides a list of those which are maximal. Our next task is to examine larger groups of homeomorphisms which also define unique classes of groups by considering extensions of the groups we have found. Rather than examine all the different classes of groups, we restrict our attention to normal extensions of genus $0$ groups for which the corresponding epimorphism $\eta$ is unique up to the action of $\aut{(A)}$. We focus on this case both as a case which is not computationally overwhelming, and also because there is a wealth of knowledge regarding important subfamilies of such surfaces, for example hyperelliptic surfaces and more generally, cyclic $p$-gonal surfaces (see for example \cite{Woo2}). See also the discussion preceding Proposition \ref{prop-uniquevector}. Before considering these groups in detail, using the results of \cite{BroWo}, we can determine which groups and signatures yield groups for which the corresponding epimorphism $\eta$ is unique up to the action of $\aut{(A)}$.

\begin{theorem}
\label{thm-uniqueclasses}
The only choices of signature $(\rho ;p^r)$ and positive integer $n$ for which there is a unique class of maximal group actions of a $p$-rank $n$ elementary abelian group $A$ on a surface of genus $$\sigma =1+p^n (\rho -1) +\frac{rp^{n-1}(p-1)}{2} $$ and such that any surface kernel epimorphism $\eta \colon \Gamma (\rho;p^r) \rightarrow A$ is unique up to $\aut{(A)}$ are the following:

\begin{enumerate}

\item
$(0;p^r)$, $n=r-1$ for any $r>1$ and any $p$

\item
$(\rho;-)$, $n=2\rho$ for any $\rho \geq 2$

\item
$(\rho;p^r)$, $n=2\rho+r-1$ for any $\rho \geq 2$

\item
$(0 ;5^3)$, $n=1$

\item
$(0 ; 2,\dots ,2)$ where $r$ is even and $n=1$

\end{enumerate}
\end{theorem}

As remarked above, we shall only be considering the purely ramified cases ($1$, $4$ and $5$). Our main goal is to show that if $N$ is a group of homeomorphisms which is a normal extension of  $A$ with signature $(0;p^{r})$ given above, then $N$ is unique up to topological equivalence. In order to do this, we shall use the correspondence between Fuchsian groups and automorphism groups of surfaces.

We fix some notation. Let $A$ denote an elementary abelian group and let $\Gamma_{p}$ denote a Fuchsian group with signature $(0;p^r)$ with the signature and group $A$ satisfying one the unramified cases of Theorem \ref{thm-uniqueclasses}. Let $\eta \colon \Gamma_{p} \rightarrow A$ denote a surface kernel epimorphism,  let $\Pi$ denote the kernel of $\eta$ and let $S$ be the surface ${\mathbb H} /\Pi$ (so $A$ acts on $S$). Let $N$ denote a normal extension of $A$ which also acts on $S$ and let $\Gamma$ be the Fuchsian group such that $\eta_{N} \colon \Gamma \rightarrow N$ is a surface kernel epimorphism with kernel $\Pi$, $\Pi \leq \Gamma_{p} \leq \Gamma$ and $\eta_{N} |_{\Gamma_{p}} =\eta$. Finally, let $K=\Gamma /\Gamma_{p}$ and let $\chi \colon \Gamma \rightarrow K$ denote the quotient map. The following Lemma allows us to manipulate sets of canonical generators for genus zero groups.  

\begin{lemma}
\label{lemma-canonical}
Let $\Gamma$ and $\Gamma^{\prime}$ be two genus zero Fuchsian groups with signatures $(0;m_{1} ,\dots ,m_{r})$ and $(0;m_{1}^{\prime} ,\dots ,m_{r}^{\prime})$, and sets of canonical generators  $C_{1} ,\dots ,C_{r}$ and $C_{1}^{\prime} ,\dots ,C_{r}^{\prime}$, respectively. For some $i$  satisfying  $1\leq i <r$ assume that $m_{i}^{\prime}= m_{i+1}$,  $m_{i+1}^{\prime}= m_{i}$, and $m_{j}^{\prime}= m_{j}$ otherwise. Then the map $\gamma : \Gamma \rightarrow \Gamma^{\prime} $ defined by

$$\gamma \colon \begin{cases} C_{i} \rightarrow C_{i+1}^{\prime} &  \\ C_{i+1} \rightarrow \left(C_{i+1}^{\prime}\right)^{-1} C_{i}^{\prime} C_{i+1}^{\prime} &  \\ C_{j} \rightarrow C_{j}^{\prime} & i\neq i,i+1 \end{cases}$$ is an isomorphism. Also the images $\gamma(C_{1}),\dots,\gamma(C_{r})$ are a set of canonical  $(0;m_{1} ,\dots ,m_{r})$ generators of $\Gamma^{\prime}$.
\end{lemma}

\begin{remark}
\label{remark-canonical} 
We are particularly interested in the case where $\Gamma = \Gamma^{\prime}$ and $C_{i} =C_{i}^{\prime}$ for all $i$.  By repeated application of the types of transformation above on a set of canonical  $(0;m_{1} ,\dots ,m_{r})$ generators of $\Gamma$ we can create a set of canonical  $(0;m_{1}^{\prime} ,\dots ,m_{r}^{\prime})$ generators of $\Gamma$ where  $(m_{1}^{\prime} ,\dots ,m_{r}^{\prime})$ is any permutation of $(m_{1} ,\dots ,m_{r})$. Moreover, there is an automorphism $\gamma$ of $\Gamma$ carrying the given set of generators onto the final transformed set of generators.
\end{remark}

\begin{lemma}
\label{lemma-K}
Let $\Pi \leq \Gamma_{p} \leq \Gamma$  and $K=\Gamma /\Gamma_{p}$ be as above and $C_{1} ,\dots ,C_{r}$ a set of canonical generators of $\Gamma$. Then the group $K$ is isomorphic to one of the following groups: $C_{n}$, (cyclic of order $n$), $D_{n}$ (dihedral of order $n$), $A_{4}$, $S_{4}$, or $A_{5}$. Moreover: $K=C_{n}$ if and only if precisely two canonical generators of $\Gamma$ have non-trivial image under $\chi$, the order of both these images being $n$; and $K=D_{n}$, $A_{4}$, $S_{4}$ or $A_{5}$ respectively if and only of precisely three canonical generators of $\Gamma$ have non-trivial image under $\chi$, the orders of these images being $2$, $2$, $n$ for $D_{n}$, $2$, $3$, $3$ for $A_{4}$, $2$, $3$, $4$ for $S_{4}$, and $2$, $3$, $5$ for $A_{5}$. 
\end{lemma}

\begin{proof}
For details, see Proposition 4.1 of \cite{Woo2}.
\end{proof}

\begin{remark}
\label{remark-images}
We note that since $\Gamma_{p}$ has signature $(0;p^r)$, all torsion elements in $\Gamma_{p}$ have order $p$. In particular, all canonical generators of $\Gamma$ with trivial image under $\chi$ must have order $p$ and if $C_{i}$ is a canonical generator of $\Gamma$ with non-trivial image of order $a$ under $\chi$, then $C_{i}$ must have order $a$ or $ap$. We also note that since $\Gamma_{p}$ has orbit genus $0$, so must $\Gamma$.

\end{remark}

\begin{definition}
If $\Gamma$ is a normal extension of $\Gamma_{p}$ with quotient group $K$, we call the homomorphism $\chi \colon \Gamma \rightarrow \Gamma /\Gamma_{p}$ a $K$-epimorphism.
\end{definition}

We shall prove our main result through a series of Lemmas.

\begin{lemma}
Suppose $\chi_{1} ,\chi_{2} \colon \Gamma \rightarrow K$ are $K$-epimorphisms. Then the kernels of $\chi_{1}$ and $\chi_{2}$ satisfy $\Ker(\chi_{1})=\Ker(\chi_{2} )$ if and only if $O(\chi_{1} (C_{i})) =O(\chi_{2} (C_{i})) $ for each canonical generator $C_{i}$ (where $O$ denotes the order of an element).
\end{lemma}

\begin{proof}
Clearly if $\Ker(\chi_{1} )=\Ker(\chi_{2} )$ then we must have $O(\chi_{1} (C_{i})) =O(\chi_{2} (C_{i})) $ for each canonical generator $C_{i}$. To prove the converse, it suffices to prove that if $O(\chi_{1} (C_{i})) =O(\chi_{2} (C_{i})) $ for each canonical generator $C_{i}$, then there exists $\alpha \in \aut{(K)}$ such that $\chi_{1} =\alpha \circ \chi_{2}$. The converse follows from the well-known uniqueness of the $K$-group actions on the sphere; but for completeness we supply the details.

Let  $C_{i}$, $C_{j}$, $C_{k}$, $i < j < k$  (just   $C_{i}$, $C_{j}$ for $K=C_{n}$) be the canonical generators with nontrivial images. Let $x_{1}=\chi_{1}(C_{i})$, $y_{1}=\chi_{1}(C_{j})$, $z_{1}=\chi_{1}(C_{k})$. Then, $x_{1}y_{1}z_{1}=1$, and hence $(x_{1},y_{1},z_{1})$ is a $(a_{1},a_{2},a_{3})$-generating vector of $K$ where  $(a_{1},a_{2},a_{3})$ is a permutation of the orders listed in Lemma \ref{lemma-K} (an $(n,n)$-vector if  $K=C_{n}$). Define  $(x_{2},y_{2},z_{2})$ similarly. The group $\aut{(K)}$ acts without fixed points on the set of generating vectors. If  $(x_{1},y_{1},z_{1})$ and  $(x_{2},y_{2},z_{2})$ are equivalent under  $\aut{(K)}$ then there is $\alpha \in \aut{(K)}$ such that $\chi_{1} =\alpha \circ \chi_{2}$. It therefore  suffices to show there are exactly $|\aut{(K)}|$ generating vectors in the five different cases. The case of $K=C_{n}$ is straightforward. The number of $(n,n)$-vectors is  $\phi (n)= |\aut (C_{n})|$  since a vector is determined by the first entry which must be a generator. For the remaining cases, the number of $(a_{1},a_{2},a_{3})$-generating vectors can be  calculated by the character formula given in \cite[Theorem 3]{Jones1}.  In every case the number of generating vectors equals $|\aut (K)|$.
\end{proof}

\begin{lemma}
\label{lemma-Kepi}
Suppose $\chi_{1} ,\chi_{2} \colon \Gamma \rightarrow K$ are $K$-epimorphisms. Then there exists $\gamma \in \aut{(\Gamma )}$ such that $\Ker(\chi_{1} )=\Ker(\chi_{2} \circ \gamma)$.
\end{lemma}

\begin{proof}
By the previous lemma, it suffices to show that there exists $\gamma \in \aut{(\Gamma )}$  such that $O(\chi_{1} (C_{i})) =O(\chi_{2} \circ \gamma (C_{i}))$ for each canonical generator $C_{i}$. We shall prove the result assuming that $K\neq C_{n}$ (so precisely three canonical generators have non-trivial image under a $K$-epimorphism) - the proof for $K=C_{n}$ is similar.

By Lemma \ref{lemma-canonical} and Remark \ref{remark-canonical}, without loss of generality, we may choose two sets of canonical generators $C_{1} ,\dots , C_{r}$ and $C_{1}' ,\dots , C_{r}'$ with $O(C_{i}) =O(C_{i}')$ for all $i$ and such that $C_{1},C_{2},C_{3}$ have non-trivial image under $\chi_{1}$ and $C_{1}',C_{2}',C_{3}'$ have non-trivial image under $\chi_{2}$. Moreover, by applying $\gamma \in \aut{(\Gamma )}$ defined by $\gamma (C_{i}')=C_{i}$, we may in fact assume that $C_{1},C_{2},C_{3}$ have non-trivial image under both $\chi_{1}$ and $\chi_{2}$.

Through our choice of $C_{1},\dots ,C_{r}$, it is clear that $O(\chi_{1} (C_{i})) =O(\chi_{2} (C_{i}))$ for all $i\geq 4$, so we need to examine $i=1,2,3$. First suppose that $O(\chi_{1} (C_{1})) \neq O(\chi_{2} (C_{1}))$. By Remark \ref{remark-images}, this can only happen if $O(C_{1})=ap$ and either $O(\chi_{1} (C_{1})) =a$ and $O(\chi_{2} (C_{1}))=ap$ or $O(\chi_{1} (C_{1})) =ap$ and $O(\chi_{2} (C_{1}))=a$ for some integer $a$. Without loss of generality, we assume that $O(\chi_{1} (C_{1})) =ap$ and $O(\chi_{2} (C_{1}))=a$.

Since $O(\chi_{1} (C_{1})) =ap$ and $O(\chi_{2} (C_{1}))\neq ap$, it follows that either $O(\chi_{2} (C_{2})) =ap$ or $O(\chi_{2} (C_{3})) =ap$. Without loss of generality (using Remark \ref{remark-canonical} if necessary) we assume that $O(\chi_{2} (C_{2})) =ap$. Since $O(\chi_{2} (C_{2})) =ap$, Remark \ref{remark-images} implies $O(C_{2}) =ap$ or $ap^2$. However, if $O(C_{2})=ap^2$, then under any $K$-epimorphism $\chi$, we would have $O(\chi (C_{2})) =ap$ or $O(\chi (C_{2})) =ap^2$. In particular, under $\chi_{1}$, both $C_{1}$ and $C_{2}$ would have order divisible by $ap$, and through observation of the possible orders given in Lemma \ref{lemma-K}, this is not possible. Thus we have $O(C_{2}) =ap$ and in particular, $O(C_{1}) =O(C_{2})$.

Now since   $O(C_{1}) =O(C_{2})$, it follows that there is an automorphism $\gamma$  of $\Gamma$ such that  $O(\chi_{1} (C_{i})) = O(\chi_{2} \circ \gamma (C_{i}))$ for $i=1,4,\ldots,r$. If $O(\chi_{1} (C_{2})) = O(\chi_{2} \circ \gamma (C_{2}))$, then we must have $O(\chi_{1} (C_{3})) = O(\chi_{2} \circ \gamma (C_{3}))$ and the result follows. If $O(\chi_{1} (C_{2})) \neq O(\chi_{2} \circ \gamma (C_{2}))$, then by the above $O(\chi_{2} \circ \gamma (C_{2}))=a$ and we know $O(C_{2})=ap$, so it follows that $O(\chi_{1} (C_{2}))=ap$. However, this would imply that $O(\chi_{1} (C_{2}))=ap=O(\chi_{1} (C_{1}))$, and through observation of the possible orders given in Lemma \ref{lemma-K}, this is not possible. Thus  $O(\chi_{1} (C_{2})) = O(\chi_{2} \circ \gamma (C_{2}))$ and consequently $O(\chi_{1} (C_{3})) = O(\chi_{2} \circ \gamma (C_{3}))$ and the result follows.
\end{proof}

We are now ready to prove our main result.
\begin{theorem}
\label{thm-pureramifiednormal}
Suppose that $N$ is a normal extension of $A$ with signature $(0;p^r)$ by the group $K$ satisfying one the unramified cases of Theorem \ref{thm-uniqueclasses}. Then $N$ defines a unique topological equivalence class of homeomorphisms of a surface of genus $\sigma$.
\end{theorem}

\begin{proof}
Suppose that $\eta_{1} ,\eta_{2} \colon \Gamma \rightarrow N$ are two surface kernel epimorphisms and let $\Pi_{1}$ and $\Pi_{2}$ denote the kernels respectively. We need to show that there exists $\gamma \in \aut{(\Gamma )}$ and $\alpha \in \aut{(N)}$ such that $\alpha \circ \eta_{2} \circ \gamma =\eta_{1}$. Let $\chi_{1}$ and $\chi_{2}$ denote the corresponding $K$-epimorphisms obtained by composing $\eta_{1}$ and $\eta_{2}$ with the quotient map $N\rightarrow N/A$ and let $\Gamma_{p,1}$, $\Gamma_{p,2}$ denote the preimages of $A$ under $\eta_{1}$ and $\eta_{2}$ respectively. Then we have the partial lattice of subgroups and quotient groups of $\Gamma$ given in Figure \ref{figure-groups} (where $i$ denotes inclusion of subgroups).

\begin{figure}[h]
\label{figure-groups}

\begin{center}
$
\xymatrix @R=.4in @C=.3in {
K & N \ar[d] \ar[l] & & \ar@/^-2pc/[lll]_{\chi_{1}} \ar@/^2pc/[rrr]^{\chi_{2}} \ar[ll]_{\eta_{1}} \ar[rr]^{\eta_{2}} \Gamma &  &  N \ar[d] \ar[r] & K  \\
& A & \ar[l]^{\eta_{1}} \Gamma_{p,1}  \ar@{^{(}->}[ur]^{i} &  &   \Gamma_{p,2}  \ar@{^{(}->}[ul]_{i}  \ar[r]_{\eta_{2}} & A & \\
& & \Pi_{1}  \ar@{^{(}->}[u]^{i} &  & \Pi_{2}  \ar@{^{(}->}[u]_{i} & & \\
}
$
\end{center}

\caption{\label{figure-groups}Groups, Quotients and Quotient Maps}

\end{figure}
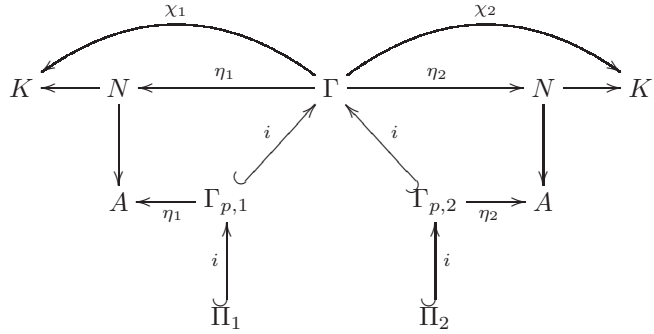

By Lemma \ref{lemma-Kepi}, there exists $\gamma \in  \aut{(\Gamma )}$ such that $\Ker(\chi_{1} )=\Ker(\chi_{2} \circ \gamma)$ and consequently, $\gamma (\Gamma_{p,1} )=\Gamma_{p,2}$. Consider the group $\gamma (\Pi_{1} )$. Since $\gamma (\Pi_{1} )$ is torsion free, $\gamma (\Pi_{1} )\vartriangleleft \Gamma_{p,2}$ and $\Gamma_{p,2} /\gamma (\Pi_{1} ) =A$, by the uniqueness of $\Pi_{2} \leq \Gamma_{p,2}$ with these properties, we must have $\gamma (\Pi_{1} )=\Pi_{2}$. In particular, $\Ker(\eta_{1} )=\Ker(\eta_{2} \circ \gamma )$, so there exists $\alpha \in \aut{(N)}$ such that $\eta_{1} =\alpha \circ \eta_{2} \circ \gamma$. Since this argument holds for any $\eta_{1}$ and $\eta_{2}$, it follows that all such surface kernel epimorphisms from $\Gamma$ to $N$ are equivalent under the action of $\aut{(N)} \times \aut{(\Gamma )}$, and thus there exists a group of homeomorphisms which is unique up to topological equivalence isomorphic to $N$ with the same signature as $N$ containing $A$.
\end{proof}

The following are interesting consequences of Theorem \ref{thm-pureramifiednormal}.

\begin{corollary}
Any finite group of homeomorphisms $G$ of a surface of genus $\sigma \geq 2$ which contains a hyperelliptic involution is unique up to topological equivalence.

\end{corollary}

\begin{proof}
This follows from the fact that the group generated by the hyperelliptic involution is normal in any finite group of automorphisms and that the hyperelliptic involution is precisely Case $5$ of Theorem \ref{thm-uniqueclasses}.
\end{proof}

\begin{corollary}
Suppose $A=C_{5}$ is a group of homeomorphisms of a surface $S$ of genus $\sigma =2$ and with signature $(0 ;5^3)$. Then $A$ is contained in a unique finite group of homeomorphisms $N=C_{10}$ of $S$ which is unique up to topological equivalence.
\end{corollary}

\section{\label{sec-hyperfermat}Hyper-Fermat curves}
In Theorem \ref{thm-uniqueclasses} there are only two infinite families of genus $0$  actions of elementary abelian groups $A$ with action unique up to $\aut{(A)}$. One family is the set of hyperelliptic curves whose defining equation are trivially constructed by definition.  In this section we give explicit geometric constructions of the curves corresponding to the other family.

\subsection{Construction}

Let $X=(x_{0},\cdots ,x_{n})$ be a point in $\mathbb{C}^{n+1}-\{0\}$ and $%
\overline{X}=(x_{0}:\cdots :x_{n})\in \mathbb{P}^{n}$ be the point
in
projective space determined by $X$ given in homogeneous coordinates. Let $%
U_{p}$ be the group of $p$'th roots of unity, let
$A_{n+1}=U_{p}^{n+1},$ let $Z_{n+1}\subset A_{n+1}$ be the scalars
$\{(a,\ldots ,a):a\in U_{p}\},$ and
let $\overline{A_{n}}=$ $A_{n+1}/$ $Z_{n+1}.$ The group $A_{n+1}$ acts on $%
\mathbb{P}^{n}$ via $(a_{0},\ldots ,a_{n})\cdot (x_{0}:\cdots
:x_{n})\rightarrow (a_{0}x_{0}:\cdots :a_{n}x_{n}).$ The kernel of
this
action is $Z_{n+1}$ so $\overline{A_{n}}$ acts effectively on $\mathbb{P}%
^{n}.$ The action has fixed points as follows. Let $H_{i}$ $\subset \mathbb{P%
}^{n}$ be the hyperplane defined by $x_{i}=0.$ Then $H_{i}$ is fixed by
the
subgroup $B_{i}$ of order $p$ in $\overline{A_{n}}$ which is the image in $%
\overline{A_{n}}$ of $\{(1,\ldots a_{i},\ldots 1):a_{i}\in
U_{p}\}$. Next
define the map $q:$ $\mathbb{P}^{n}\rightarrow \mathbb{P}^{n}$ by $%
(x_{0}:\cdots :x_{n})\rightarrow (x_{0}^{p}:\cdots :x_{n}^{p}).$
Observe that $q$ is an $\overline{A_{n}}$-equivariant branched
cover whose fibers are exactly the $\overline{A_{n}}$ orbits.
Finally, let $T$ be a generic line in $\mathbb{P}^{n}$ that does
not lie in any $H_{i}$ nor meet any intersection $H_{i}\cap H_{j}$, and let $S=q^{-1}(T).$ We will call $S$  a \emph{hyper-Fermat
curve}.  We are going to show that hyper-Fermat curves have an $\overline{A_{n}}$ action with signature $(0;p^{n+1})$, and that every such curve is isomorphic to such a hyper-Fermat curve. The standard Fermat curve is a plane curve with equation $x^p+y^p=z^p$ and $\overline{A_{2}}$ action.

\begin{remark}
Typically one would like a plane equation for a defining equation of a curve. The authors tried to find such equations for hyper-Fermat curves and were successful, for small $n$ and $p$, using the projection to $\mathbb{P}^{2}$ and computing the equation of the image of $S$, using Groebner basis methods. However, even the simplest resulting equations were so complex that they are not worth recording here. The given construction in $\mathbb{P}^{n}$  has the virtue that the action is linear. It is also clear that the given construction has minimal  dimension with a linear action.
\end{remark}

Before stating our main proposition on hyper-Fermat curves let us give an explicit way to construct lines  $T$ that satisfy the required hypotheses. The line $T$ is given by the system of equations $CX=0$ where $C$ is an $(n-1)\times (n+1)$ matrix. The following lemma shows precisely when $T$ satisfies the hypotheses.

\begin{lemma}
\label{lemma-Tconstr}
Suppose that the line  $T$ in $\mathbb{P}^{n}$ is defined by the set of equations $CX=0$  where $C$ is an $(n-1)\times (n+1)$ matrix, and let  $H_{i}$ be a coordinate hyperplane as previously defined.   Then, $T$ does not lie in any $H_{i}$, nor meet any intersection $H_{i}\cap $
$H_{j}$, if and only if for every  submatrix $C^{\prime \prime }$, obtained from  $C$ by deleting two distinct columns from $C$, the submatrix $C^{\prime \prime }$ is invertible. The set of lines satisfying these conditions form an open dense subset in the Grassman manifold of lines in $\mathbb{P}^{n}$.
\end{lemma}
 \begin{proof}
 Let $\overline{X}$ lie in  $T$. If the coordinates  $x_{i},x_{j}$ of  $X$ satisfy  $x_{i}=x_{j}=0$ for distinct $i,j$, then $CX=0$ implies that all the remaining coordinates are zero, otherwise $T\cap H_{i}\cap H_{j}$ would be non-empty. This implies that the submatrix $C^{\prime \prime }$ obtained from  $C$ by deleting the $i$'th and $j$'th columns from $C$ has a trivial nullspace and hence is invertible. Arguing in the other direction, if $C^{\prime \prime }$  is invertible, then  $T\cap H_{i}\cap H_{j}$ is empty.

Next, let us show that the conditions on  $C$ implies that $T$ does not lie in any of the hyperplanes $H_{i}$. Let $D_{i}$ be the $1\times (n+1)$ matrix
whose only non-zero entry is $d_{i}=1.$ Then $H_{i}$ is defined by
$D_{i}X=0$, and the equation for the set of points in $T\cap H_{i}$ is
\begin{equation*}
C_{i}^{\prime }X=0,C_{i}^{\prime }=\left[
\begin{array}{c}
C \\
D_{i}%
\end{array}%
\right].
\end{equation*}%
The set  $T\cap H_{i}$ is a singleton if and only if the rank of $C_{i}^{\prime }$ is $n-1$. But the rank is $n-1$ because of the constraints on $C.$ To
see this, we do the following. Using row operations, zero out all
entries in column $i$ of $C_{i}^{\prime }$ except the last row.
Remove any other column, say column $j$, and permute the columns of
the resulting matrix so that column $i$ is the first column. The
resulting matrix has the form
\begin{equation*}
\left[
\begin{array}{cc}
0 & C^{\prime \prime } \\
1 & 0%
\end{array}%
\right]
\end{equation*}%
where $C^{\prime \prime }$ is obtained from $C$ by removing columns $i$ and $%
j$ from $C.$

The set of matrices satisfying the conditions form an
open set in the vector space of all $(n-1)\times (n+1)$ matrices. The subset will be dense if it is non-empty. In Example \ref{ex-vand1} below we construct an example of such a matrix and so the set is non-empty. It is easy to show that the projection of this open set to the Grassman is open and dense.
\end{proof}

\begin{theorem}
Let $S$ be a hyper-Fermat curve. Then $S$ is smooth, irreducible curve of genus $%
\sigma =1+p^{n-1}\frac{(n-1)p+n+1}{2}.$ The elementary abelian group $%
\overline{A_{n}}$ of order $p^{n}$ acts on $S$ with signature
$(0;p^{n+1})$ and $q:S\rightarrow T$ is the quotient map. The
$n+1$ points determined by the intersections $T\cap H_{i}$ are the
branch points of $q.$ Furthermore, any smooth curve, with an $\overline{A_{n}}$ action with signature $(0;p^{n+1})$ is conformally equivalent to  a hyper-Fermat curve.
\end{theorem}

\begin{proof}
The last statement follows from Example \ref{ex-vand2}.
 Assume that  $T$ is defined by a matrix $C$ as described in Lemma \ref{lemma-Tconstr}. First we show that  $S$ is smooth.
If $C_{i}$ is the $i$'th row of $C$ then $T$ is the intersection
of the hyperplanes $\cap K_{i}$ where each $K_{i}$ is given by $C_{i}X=0.$ The surface $S$ is
the intersection $S=$ $\bigcap\limits_{i}q^{-1}(K_{i}).$ Each
$q^{-1}(K_{i})$ is a smooth hypersurface given by the set

\begin{equation*}
\{\overline{X}=(x_{0}:\cdots :x_{n})\in \mathbb{P}%
^{n}:f_{i}(x_{0},\ldots ,x_{n})=c_{i,0}x_{0}^{p}+\cdots
+c_{i,n}x_{n}^{p}=0\}.
\end{equation*}%
If we can show that the normals $\nabla f_{i}$ of the
$q^{-1}(K_{i})$ are linearly independent at each point of $S,$
then $S$ will be the transverse intersection of smooth
hypersurfaces and hence it will be smooth itself. Arrange
the $\nabla f_{i}$ into a matrix $G$ of the form
\begin{equation*}
G=p\left[
\begin{array}{ccc}
c_{1,0}x_{0}^{p-1} & \cdots & c_{1,n}x_{n}^{p-1} \\
\vdots & \ddots & \vdots \\
c_{n-1,0}x_{0}^{p-1} & \cdots & c_{n-1,n}x_{n}^{p-1}%
\end{array}%
\right]
\end{equation*}%
Then the gradients will be linearly independent if two columns of
$G$ can be deleted leaving a non singular $(n-1)\times (n-1)$
submatrix $G^{\prime}$. By the constrains on $T$, at most one of the
coordinates $x_{i}$ is zero. So we may assume, for instance, that
$x_{2},\ldots ,x_{n}$ are nonzero. But then, upon deleting the
first two columns of $G$ and computing determinants we get
\begin{eqnarray*}
\det (G^{\prime})&=&p^{n-1}x_{2}^{\left( p-1)(n-1\right) }\cdots x_{n}^{\left(
p-1)(n-1\right) }\\
& & \times \det \left( \left[
\begin{array}{ccc}
c_{1,2} & \cdots & c_{1,n} \\
\vdots & \ddots & \vdots \\
c_{n-1,2} & \cdots & c_{n-1,n}%
\end{array}%
\right] \right)
\end{eqnarray*}%
By the constraints on $C$ this determinant is non-zero, and hence
$S$ is smooth.

To prove that $S$ is connected and hence irreducible, we use a
monodromy argument. Let $T^{\circ }$ be the projective line $T$
with the intersections with the coordinate hyperplanes removed and let $S^{\circ }=q^{-1}(T^{\circ })$. Then by construction $q:$
$S^{\circ }\rightarrow T^{\circ }$ is an unramified covering space
each of whose fibers is a full $\overline{A_{n}}$ orbit. If we can
show that the monodromy action of $\pi _{1}(T^{\circ })$ is
transitive on the fibers then $S^{\circ }$ will be connected as
$T^{\circ }$ is connected. This implies that $S$ is connected. Let
$Y_{i}$ $\in T-T^{\circ }$ be the unique point of intersection of
$H_{i}$ and $T.$ By construction of the $i$'th coordinate of
$\overline{Y_{i}}$ is the only zero coordinate. Let $V$ be a direction vector on the line $T$ and let $%
\alpha (t)=Y_{i}+re^{2\pi it}V$ for suitably chosen $r.$ Any lift $%
\widetilde{\alpha }(t)$ to $S$ is given by
\begin{equation*}
\widetilde{\alpha }(t)=\left( a_{0}\sqrt[p]{y_{0}+re^{2\pi
it}v_{0}}:\cdots :a_{n}\sqrt[p]{y_{n}+re^{2\pi it}v_{n}}\right)
\end{equation*}%
where $(a_{0},\ldots a_{n})\in A_{n+1}$. By selecting $r$
sufficiently small we can ensure that $\beta
_{j}(t)=\sqrt[p]{y_{j}+re^{2\pi it}v_{j}}$ defines a closed loop
in the plane if $i\neq j.$ On the other hand the loop $\beta
_{i}(t)=$ $\sqrt[p]{y_{i}+re^{2\pi it}v_{i}}=\sqrt[p]{re^{2\pi
it}v_{i}}$ satisfies $\beta (1)=e^{2\pi i/p}\beta (0).$ It follows
then that for any lift of $\widetilde{\alpha }(t)$ that
\begin{equation*}
\widetilde{\alpha }(1)=(1,\ldots ,a_{i},\ldots 1)\cdot \widetilde{\alpha }%
(0),\text{ }a_{i}=e^{2\pi i/p}.
\end{equation*}%
Thus the local monodromy at the puncture $Y_{i}$ generates the subgroup $%
B_{i}.$ Since these subgroups generate
$\overline{A_{n}}$, it follows that the monodromy is transitive on the fibers.
Note that since $\overline{A_{n}}$ is abelian we don't have to
worry about the base point in monodromy calculations. Finally we
observe that $S\rightarrow S/\overline{A_{n}}$ is branched over
$n+1$ points and the stabilizer of each of these points are one of the cyclic groups $B_{i}$ of
order $p$. It follows from the Riemann-Hurwitz equation that
\begin{eqnarray*}
\frac{2(\sigma -1)}{p^{n}} &=&(-2+(n+1)(1-\frac{1}{p}) \\
\sigma &=&1+p^{n-1}\frac{p(n-1)+n+1}{2}
\end{eqnarray*}
\end{proof}

\begin{example}
\label{ex-vand1}
Let $w_{0},\ldots ,w_{n}$
be $n+1$ distinct complex numbers and let $C$ be a modified
Vandermonde matrix.
\begin{equation*}
C=\left[
\begin{array}{cccc}
1 & 1 & \cdots & 1 \\
w_{0} & w_{1} & \cdots & w_{n} \\
\vdots & \vdots & \ddots & \vdots \\
w_{0}^{n-2} & w_{1}^{n-2} & \cdots & w_{0}^{n-2}%
\end{array}%
\right]
\end{equation*}%
obtained by removing the last two rows of a standard Vandermonde
matrix. Then the matrix $C$ satisfies the required conditions given in Lemma \ref{lemma-Tconstr}. This is immediate since removing two columns leaves an invertible standard Vandermonde matrix.
\end{example}

\subsection{Moduli}

According to the Proposition \ref{prop-uniquevector} and Theorem \ref{thm-uniqueclasses} two curves $S_{1}$ and $S_{2}$ with $\overline{A_{n}}$ action and signature $(0;p^{n+1})$ will be conformally equivalent if the quotients $S/\overline{A_{n}}$ are conformally equivalent taking branch points into account. The quotients are the spheres with $n+1$ branch points. To determine when two hyper-Fermat curves are conformally equivalent and to show that all curves with the given  $\overline{A_{n}}$-action are equivalent to hyper-Fermat curves,  it will be useful match the branch points on $T$ with points on the sphere. We want to parameterize $T$ by a map $\varphi :\mathbb{P}%
^{1}\rightarrow \mathbb{P}^{n}$ with $T\cap H_{i}=\{\varphi
(\lambda _{i})\}$
where $\lambda _{0},\ldots ,\lambda _{n}$ are finite distinct points in $%
\mathbb{P}^{1}$. The $\lambda _{i}$ should have formulae dependent
on the matrix $C$. Once $\lambda _{0},\lambda _{1},\lambda _{2}$
are fixed the remaining points are determined.

To this end let
$Q_{i}\in \mathbb{C}^{n+1}$
be such that $T\cap H_{i}=\{\overline{Q_{i}}\}.$ Define the map $\varphi :%
\mathbb{P}^{1}\rightarrow \mathbb{P}^{n}$ by $\varphi (s:t)=\overline{%
sP_{1}+tP_{2}},$ where $P_{1}$ and $P_{2}$ are appropriately
chosen in the span of $\left\langle Q_{0},Q_{1}\right\rangle .$
Then $\varphi (\lambda
_{i})=\varphi (\lambda _{i}:1)=\overline{\lambda _{i}P_{1}+P_{2}}.$ Writing $%
Q_{i}=c_{i}Q_{0}+d_{i}Q_{1}$ we observe that there are $u_{i}$ such that $%
c_{i}Q_{0}+d_{i}Q_{1}=Q_{i}=u_{i}(\lambda _{i}P_{1}+P_{2}),$ i.e.,
\begin{equation*}
\left[
\begin{array}{cc}
c_{i} & d_{i}%
\end{array}%
\right] \left[
\begin{array}{c}
Q_{0} \\
Q_{1}%
\end{array}%
\right] =\left[
\begin{array}{cc}
u_{i}\lambda _{i} & u_{i}%
\end{array}%
\right] \left[
\begin{array}{c}
P_{1} \\
P_{2}%
\end{array}%
\right]
\end{equation*}%
By scaling $P_{1}$ and $P_{2}$ we may assume that $u_{0}=1.$ From
the first
two equations we have%
\begin{equation*}
\left[
\begin{array}{c}
Q_{0} \\
Q_{1}%
\end{array}%
\right] =\left[
\begin{array}{cc}
\lambda _{0} & 1 \\
u_{1}\lambda _{1} & u_{1}%
\end{array}%
\right] \left[
\begin{array}{c}
P_{1} \\
P_{2}%
\end{array}%
\right] ,
\end{equation*}%
hence
\begin{equation*}
\left[
\begin{array}{cc}
c_{i} & d_{i}%
\end{array}%
\right] \left[
\begin{array}{c}
Q_{0} \\
Q_{1}%
\end{array}%
\right] =\left[
\begin{array}{cc}
u_{i}\lambda _{i} & u_{i}%
\end{array}%
\right] \left[
\begin{array}{cc}
\lambda _{0} & 1 \\
u_{1}\lambda _{1} & u_{1}%
\end{array}%
\right] ^{-1}\left[
\begin{array}{c}
Q_{0} \\
Q_{1}%
\end{array}%
\right]
\end{equation*}%
or
\begin{equation*}
\left[
\begin{array}{cc}
u_{i}\lambda _{i} & u_{i}%
\end{array}%
\right] =\left[
\begin{array}{cc}
c_{i} & d_{i}%
\end{array}%
\right] \left[
\begin{array}{cc}
\lambda _{0} & 1 \\
u_{1}\lambda _{1} & u_{1}%
\end{array}%
\right] =\left[
\begin{array}{cc}
c_{i}\lambda _{0}+d_{i}u_{1}\lambda _{1} & c_{i}+d_{i}u_{1}%
\end{array}%
\right]
\end{equation*}%
thus
\begin{equation*}
\lambda _{i}=\frac{c_{i}\lambda _{0}+d_{i}u_{1}\lambda _{1}}{c_{i}+d_{i}u_{1}%
}
\end{equation*}%
Setting $i=2$ and solving for $u_{1}$ we get%
\begin{equation*}
u_{1}=\frac{c_{2}\left( \lambda _{0}-\lambda _{2}\right)
}{d_{2}\left( \lambda _{2}-\lambda _{1}\right) }
\end{equation*}%
If one of $\lambda _{0},\lambda _{1},\lambda _{2}$ is infinite the
resulting formula is obtained by taking limits. In particular for
$\lambda _{0}=0,\lambda _{1}=1,\lambda _{2}=\infty $ we get
\begin{equation*}
\lambda _{i}=\frac{-d_{i}c_{2}}{c_{i}d_{2}-d_{i}c_{2}}
\end{equation*}%
Instead of computing all the $c_{i}$ and $d_{i},$ we can compute
$P_{1}$ and $P_{2}$ from
\begin{eqnarray*}
\left[
\begin{array}{c}
P_{1} \\
P_{2}%
\end{array}%
\right]  &=&\left[
\begin{array}{cc}
\lambda _{0} & 1 \\
u_{1}\lambda _{1} & u_{1}%
\end{array}%
\right] ^{-1}\left[
\begin{array}{c}
Q_{0} \\
Q_{1}%
\end{array}%
\right]  \\
&=&\left[
\begin{array}{c}
\frac{1}{\lambda _{0}-\lambda _{1}}Q_{0}-\frac{1}{u_{1}\left(
\lambda
_{0}-\lambda _{1}\right) }Q_{1} \\
-\frac{\lambda _{1}}{\lambda _{0}-\lambda _{1}}Q_{0}+\frac{\lambda _{0}}{%
u_{1}\left( \lambda _{0}-\lambda _{1}\right) }Q_{1}%
\end{array}%
\right]
\end{eqnarray*}%
and then as $\lambda _{i}P_{1}(i)+P_{2}(i)=0$ we get
\begin{equation}
\lambda _{i}=-P_{2}(i)/P_{1}(i)=\frac{\lambda
_{1}u_{1}Q_{0}(i)-\lambda _{0}Q_{1}(i)}{u_{1}Q_{0}(i)-Q_{1}(i)}
\notag
\end{equation}%
This way only $Q_{0},Q_{1},Q_{2},c_{2}$ and $d_{2}$ need to be
calculated.

\begin{example}
\label{ex-vand2}
Let us use the procedure above for the Vandermonde example given in Example \ref{ex-vand1}. Choosing $\lambda _{0}=w_{0},\lambda
_{1}=w_{1},\lambda _{2}=w_{2}$ we get, using Maple,
\begin{equation*}
\begin{tabular}{|l|l|l|l|l|l|l|}
\hline
& $\lambda _{0}$ & $\lambda _{1}$ & $\lambda _{2}$ & $\lambda _{3}$ & $%
\lambda _{4}$ & $\lambda _{5}$ \\ \hline $n=3$ & $w_{0}$ & $w_{1}$
& $w_{2}$ & $w_{3}$ &  &  \\ \hline $n=4$ & $w_{0}$ & $w_{1}$ &
$w_{2}$ & $w_{3}$ & $w_{4}$ &  \\ \hline $n=5$ & $w_{0}$ & $w_{1}$
& $w_{2}$ & $w_{3}$ & $w_{4}$ & $w_{5}$ \\ \hline
\end{tabular}%
\end{equation*}%

We can establish the general pattern by showing that we may choose the $%
Q_{i} $ to satisfy
\begin{eqnarray*}
Q_{i}(i) &=&0 \\
Q_{i}(j) &=&\prod\limits_{k\neq i,j}\left( w_{j}-w_{k}\right)
^{-1},i\neq j.
\end{eqnarray*}%
The formulas were suggested by exploring the first few examples
with Maple. We give the proof for $Q_{0},$ the other formulas are
similar. We need to show that
\begin{equation*}
\sum\limits_{j=1}^{n}\prod\limits_{k\neq
0,j}\frac{w_{j}^{s}}{w_{j}-w_{k}}=0
\end{equation*}%
for $s=0,\ldots ,n-2.$ While this is undoubtedly a simple
algebraic identity
we are going to use the residue theorem instead. Consider the function $%
f(z)=z^{s}\prod\limits_{k=1}^{n}\left( z-w_{k}\right) ^{-1}.$ The poles of $f$ are simple and are located at $w_{1},\ldots ,w_{n},$
and the residue at the pole s is given by
$Res(f,w_{j})=\lim_{z\rightarrow
w_{k}}(z-w_{j})f(z)=w_{j}^{s}\prod\limits_{k\neq j,0}\left(
w_{j}-w_{k}\right) ^{-1}.\ $By the residue theorem
\begin{equation*}
\sum\limits_{j=0}^{n}\prod\limits_{k\neq j,0}\frac{w_{j}^{s}}{w_{j}-w_{k}}=%
\frac{1}{2\pi i}\int\limits_{\partial \Delta _{R}}f(z)dz
\end{equation*}%
where $\Delta _{R}$ is a large disc about the origin. For large $R$, $%
\left\vert f(z)\right\vert \leq 2R^{s-n}$ and hence
\begin{equation*}
\left\vert \sum\limits_{j=0}^{n}\prod\limits_{k\neq j,0}\frac{w_{j}^{s}}{%
w_{j}-w_{k}}\right\vert \leq \lim_{R\rightarrow \infty }\left\vert \frac{1}{%
2\pi i}\int\limits_{\partial \Delta _{R}}f(z)dz\right\vert \leq
\lim_{R\rightarrow \infty }\frac{1}{2\pi }2R^{s-n}2\pi R=0
\end{equation*}%
We need to simplify the formula
\begin{equation}
\lambda _{i}=-\frac{\lambda _{1}u_{1}Q_{0}(i)-\lambda _{0}Q_{1}(i)}{%
u_{1}Q_{0}(i)-Q_{1}(i)}  \notag
\end{equation}%
As $Q_{2}=c_{2}Q_{0}+d_{2}Q_{1}$ then considering the first two
coordinates we get
\begin{equation*}
\left[
\begin{array}{cc}
0 & \prod\limits_{k\neq 1,0}\left( w_{0}-w_{k}\right) ^{-1} \\
\prod\limits_{k\neq 1,0}\left( w_{1}-w_{k}\right) ^{-1} & 0%
\end{array}%
\right] \left[
\begin{array}{c}
c_{2} \\
d_{2}%
\end{array}%
\right] =\left[
\begin{array}{c}
\prod\limits_{k\neq 0,2}\left( w_{0}-w_{k}\right) ^{-1} \\
\prod\limits_{k\neq 1,2}\left( w_{1}-w_{k}\right) ^{-1}%
\end{array}%
\right]
\end{equation*}%
\begin{eqnarray*}
c_{2} &=&\frac{\prod\limits_{k\neq 1,2}\left( w_{1}-w_{k}\right) ^{-1}}{%
\prod\limits_{k\neq 1,0}\left( w_{1}-w_{k}\right) ^{-1}}=\frac{w_{1}-w_{2}}{%
w_{1}-w_{0}} \\
d_{2} &=&\frac{\prod\limits_{k\neq 0,2}\left( w_{0}-w_{k}\right) ^{-1}}{%
\prod\limits_{k\neq 1,0}\left( w_{0}-w_{k}\right) ^{-1}}=\frac{w_{0}-w_{2}}{%
w_{0}-w_{1}}
\end{eqnarray*}%
and
\begin{equation*}
u_{1}=\frac{c_{2}\left( \lambda _{0}-\lambda _{2}\right)
}{d_{2}\left(
\lambda _{2}-\lambda _{1}\right) }=\frac{\frac{w_{1}-w_{2}}{w_{1}-w_{0}}}{%
\frac{w_{0}-w_{2}}{w_{0}-w_{1}}}\frac{w_{0}-w_{2}}{w_{2}-w_{1}}=1
\end{equation*}%
Thus
\begin{equation}
\lambda _{i}=-\frac{\lambda _{1}Q_{0}(i)-\lambda _{0}Q_{1}(i)}{%
Q_{0}(i)-Q_{1}(i)}  \notag
\end{equation}%
For $i\geq 2,$ $Q_{0}(i)=\prod\limits_{k\neq 0,i}\left(
w_{i}-w_{k}\right) ^{-1}$ and $Q_{1}(i)=\prod\limits_{k\neq
1,i}\left( w_{i}-w_{k}\right) ^{-1}$ so
\begin{equation*}
Q_{1}(i)=\frac{w_{i}-w_{0}}{w_{i}-w_{1}}Q_{0}(i).
\end{equation*}%
and so
\begin{equation*}
\lambda _{i}=\frac{w_{1}\frac{w_{i}-w_{0}}{w_{i}-w_{1}}-w_{0}}{\frac{%
w_{i}-w_{0}}{w_{i}-w_{1}}-1}=\frac{w_{1}(w_{i}-w_{0})-w_{0}(w_{i}-w_{1})}{%
(w_{i}-w_{0})-(w_{i}-w_{1})}=w_{i}.
\end{equation*}
\end{example}

\end{document}